\documentclass[11pt]{article}
\usepackage{amsmath}
\usepackage{amssymb}
\usepackage{amsthm}
\usepackage{fullpage}
\usepackage{cleveref}
\usepackage[utf8]{inputenc}
\usepackage{framed}
\usepackage{enumerate}
\usepackage{color}

\newtheorem{theorem}{Theorem}[section]

\newtheorem{proposition}{Proposition}

\newtheorem{definition}[theorem]{Definition}

\newcommand{\eps}{\varepsilon}

\newcommand{\Z}{\mathbb{Z}}

\begin{document}

\title{Arithmetic Progressions in Sumsets of Sparse Sets}
\author{Noga Alon\footnote{Department of Mathematics, Princeton University, Princeton, NJ, USA. Email: \texttt{nalon@math.princeton.edu}. Research supported in part by NSF grant DMS-1855464 and the Simons Foundation.} , 
Ryan Alweiss\footnote{Department of Mathematics, Princeton University, Princeton, NJ, USA. Email: \texttt{alweiss@math.princeton.edu}. Research supported by the NSF Graduate Fellowship (GRFP).} , 
Yang P. Liu\footnote{Department of Mathematics, Stanford University, Stanford, CA, USA. Email: \texttt{yangpliu@stanford.edu}. Research supported by the Department of Defense (DoD) through the National Defense Science and Engineering Graduate Fellowship (NDSEG) Program.} , 
Anders Martinsson\footnote{Institut f\"ur Theoretische Informatik, ETH Z\"urich, Z\"urich, Switzerland. Email: \texttt{anders.martinsson@inf.ethz.ch}.} , 
Shyam Narayanan\footnote{Department of Electrical Engineering and Computer Science, Massachusetts Institute of Technology, Cambridge, MA, USA. Email: \texttt{shyamsn@mit.edu}. Research supported by the NSF Graduate Fellowship (GRFP) and 
a Simons Investigator Award.}}
\date{\today}

\maketitle

\begin{abstract}
A set of positive integers $A \subset \Z_{> 0}$
is \emph{log-sparse} if there is
an absolute constant $C$ so that for any positive integer $x$
the sequence contains
at most $C$ elements in the interval $[x,2x)$.
In this note we study arithmetic progressions in sums
of log-sparse subsets of $\Z_{> 0}$. We prove that for any
log-sparse subsets $S_1, \dots, S_n$ of $\Z_{> 0},$ the sumset $S = S_1
+ \cdots + S_n$ cannot contain an arithmetic progression of size greater
than $n^{(1+o(1))n}.$ We also show that this is nearly tight by
proving that there exist log-sparse sets $S_1, \dots, S_n$ such that $S_1
+ \cdots + S_n$ contains an arithmetic progression of size $n^{(1-o(1)) n}.$
\end{abstract}

\section{Introduction}
Arithmetic progressions have been one of the favorite research 
topics of Ron Graham. See \cite{Gr} for a lecture he has given
on the subject. The term ``arithmetic
progression'' appears in the title of seven of his papers, and he has
written, with Andr\'as Hajnal, the proof of Szemer\'edi's Theorem
on arithmetic progressions in sets of integers of 
positive upper density \cite{Sz}. 
In the present note, dedicated to his memory,
we study the maximum possible length
of arithmetic progressions in sumsets of very sparse sets.

Waring's problem, first proven by Hilbert \cite{Hi}, 
states that there exists a function $f(k)$ so that any positive 
integer can be written as the sum of at most $f(k)$ perfect 
$k$-powers. From a crude heuristic perspective, 
the density of perfect $k$-powers makes this result 
plausible. As the number of ways to write integers between $1$ and
$n$  
as sums of $r$ perfect $k$-powers is 
asymptotically $\Theta\left(n^{r/k}\right)$ if $r$ and $k$ are
fixed, 
on average one may expect any (large) $n$ to have some such representation 
as long as $r>k.$ However, for many values of $k$ there are 
congruence obstructions, so that certain
arithmetic progressions cannot be reached by a sum of $k+1$  $k$-th 
powers.
In the literature on Waring's problem, the worst cases of the congruence
obstructons are summarized in a variable $\Gamma(k)$ and the
common belief is that any large $n$ has a representation provided
$r \geq max(k+1, \Gamma(k))$.

In this note, instead of perfect $k$-powers, we
consider sums
of sets with much lower density. Namely, we consider logarithmically
sparse sets, or sets of positive integers where the number of elements
less than $n$ grows as $\log n$ or slower rather than as a fractional
power of $n$, and their sumsets. We define a log-sparse set and the
sumset of sets formally as follows.

\begin{definition} \label{def:logsparse}
A subset $T$ of $\mathbb{Z}_{>0}$ is ($C$-)\emph{log-sparse} 
if for all positive integers $x$, $|T \cap [x,2x)| \le C$.
\end{definition}

\begin{definition} \label{def:sumset}
Given sets $S_1, S_2, \dots, S_n \subset \mathbb{Z}_{>0},$ the set $S
= S_1 + S_2 + \dots + S_n$ is the \emph{sumset} of $S_1, S_2, \dots,
S_n$ if $S$ is the set of all positive integers $x$ such that $x =
x_1 + x_2 + \dots + x_n$ for some $x_1 \in S_1, x_2 \in S_2, \dots,
x_n \in S_n$.
\end{definition}

We note that the specific constant $C$ in Definition \ref{def:logsparse} 
is not crucial as long as it is at least $2$.

Note that it is impossible for all integers to be in the sumset of $r$
log-sparse sets for any fixed $r$, as by a counting argument, such a
sumset cannot contain more than $(O(\log N))^{r}$ integers below $N.$
Here we consider the maximum possible length of
\emph{arithmetic progressions} in sumsets
of such logarithmically sparse sets.

Sizes of arithmetic progressions have been well-studied in various
cases. In particular, a well-known result of Szemer\'edi \cite{Sz} 
shows that
any subset $A$ of $\mathbb{Z}_{>0}$ 
with positive upper density
contains arbitrarily long arithmetic progressions. Even
for sparser sets, such as the set of primes, it is known that they
contain 
arbitrarily long arithmetic progressions \cite{GT}, and a well-known
conjecture due to Erd\H{o}s states that as long as $A = \{a_1, a_2,
\dots \}$ satisfies $\sum \frac{1}{a_i} = \infty,$ $A$ contains
arbitrarily long arithmetic progressions.

Arithmetic progressions in sumsets have also been studied.
Bourgain \cite{Bo} proved that when $|A|=\alpha N$
and $|B|=\beta N$ are subsets of $[N]=\{1,2, \ldots ,N\}$, 
$A+B$ must contain an arithmetic
progression of size at least $\exp(\Omega_{\alpha,\beta}(\log n)^{1/3}))$.
Bourgain's result was subsequently improved by Green \cite{Green} 
and by Croot, Laba, and Sisak \cite{CLS} to $A+B$ containing 
an arithmetic progression of size at 
least $\exp(\Omega_{\alpha,\beta}(\log n)^{1/2}))$.

Sumsets of log-sparse sets do not have positive density, but 
trivially there
do exist sparse sets containing arbitrarily long arithmetic progressions,
such as the set $S = \{2^a + b: 1 \le b \le a\},$ which contains
the $k$-term
progression $2^k+1, 2^k+2, \dots, 2^k+k$ for all $k$. This set,
of course, is not
log-sparse. This raises the following question:
does there exist an integer $n$ and $n$ log-sparse sets $S_1, S_2,
\dots, S_n$ such that the sumset $S_1 + S_2 + \dots + S_n$ contains 
arbitrarily
long arithmetic progressions?

In this note we answer this question in the negative, by showing that
for all $n \ge 1$ and all log-sparse sets $S_1, S_2, \dots, S_n$,
the maximum possible size of an arithmetic progression in 
$S_1 + S_2 + \dots +
S_n$ is at most $n^{(1 + o(1)) n}$, where the $o(1)$ term tends to $0$
as $n \to \infty$ and is independent of the choice of the sets $S_1, S_2,
\dots, S_n$. We also establish a nearly matching lower bound, by proving
that for all $n$, there exist log-sparse sets $S_1, S_2, \dots, S_n$
whose sum contains an arithmetic progression of length 
at least $n^{(1-o(1)) n}$.

This question is also motivated by the following problem from the 2009
China Team Selection Test: Prove that the set $\{2^a+3^b: a, b \ge
0\}$ has no arithmetic progression of length $40$. Note that this
set can
be written as the sum of two log-sparse sets: the set of powers of $2$
and the set of powers of $3$, so a direct corollary of our upper bound
is that the longest arithmetic progression in $\{2^a+3^b: a, b \ge 0\}$
is bounded.

Throughout this note, $\log(n)$ always denotes $\log_2(n)$, and
$[n]$ denotes the set $\{1, 2, \dots, n\}$ of the first $n$
positive integers.

\section{The Upper Bound} 

In this section, we prove an upper bound on the size of the longest
arithmetic progression in the sumset of $n$ log-sparse sets.

\begin{theorem}
\label{thm:upper}
 Let $S_1, \dots S_n$ be $C$-log-sparse sets, for any fixed $C>0$, 
and let $T$ be any arithmetic progression 
in $S=S_1+\dots+S_n$. Then $|T|\leq n^{(1+O(\lg \lg n/\lg n))n}$.
\end{theorem}

\begin{proof}
For any $x\in T$, we fix a representation $x=x_1+x_2+\dots+x_n$. 
We will bound the number of elements in $T$ by finding an efficient 
encoding for an arbitrary $x\in T$. To this end, 
let $\Delta:= \max_{y\in T} y - \min_{y\in T} y $ 
and let $\delta$ be the step-length in $T$ 
(i.e. $\delta:=\Delta/(|T|-1)$). For the fixed 
representation $x=x_1+\dots+x_n$ for any $x\in T$, we say 
that $x_i$ is large if $x_i>\Delta$, small if $x_i<\delta/2n$, 
and medium otherwise. Observe that as the sum of of all small terms 
is less than $\delta/2$, $x\in T$ is uniquely determined by the 
values of all its large and medium terms.

We can encode an arbitrary $x\in T$ as follows. First we choose 
which terms are large, medium and small. There are at most 
$3^n$ choices for this. Let $a, b$ and $c$ denote the chosen 
number of terms of each respective type.

For the $a$ large terms, we first choose their internal order from 
largest to smallest, and then choose the value of each of these 
terms in decreasing order. We claim that having fixed the order, 
there are at most $O(\log n)$ choices for each term. To see this, 
we may, without loss of generality, assume that the large 
terms and internal order are given by $x_1 \geq x_2 \geq \dots \geq x_a$. 
Having already chosen $x_1, \dots, x_{i-1}$ where $i\leq a$, 
we let $$M:=\max_{y\in T} y - x_1 + \dots + x_{i-1}.$$ Clearly we 
must choose $x_i \leq M$. On the other hand, we must also have 
$$x_1+\dots+x_{i-1} + n\cdot x_i \geq \min_{y\in T} y.$$ 
Rewriting this, using the definition of $\Delta$, we get 
$n\cdot x_i + \Delta \geq M$. Since $x_i\geq \Delta$ 
we can conclude that $(n+1)x_i \geq M$ and so any valid 
choice for $x_i$ is contained in $S_i \cap [M/(n+1), M]$. 
Thus by log-sparseness there are at most $O(\log n)$ options, 
as desired. So in total, we have $O(n \log n)^a$ choices for the 
large terms.

For each medium term $x_i$, we know that it is contained in 
$S_i \cap [\delta/2n, \Delta]$, where the lower and upper bounds 
differ by a factor $2n(|T|-1)$. Thus again by log-sparseness, 
there are at most $O(\log n + \log |T|)$ options for each. 
So in total $O(\log n + \log |T|)^b$ possibilities.

Combining this, we conclude that $$|T| \leq 3^{n} \cdot 
O\left(\max\left( n \log n, \log n + \log |T|\right)\right)^n 
= O\left( n \log n + \log |T|\right)^n.$$ But this cannot 
hold if $|T|$ is too large. Assuming $|T|=(n f(n))^n$ 
where $f(n)\geq 1$ yields $f(n) \leq O\left( \log n 
+ \log f(n)\right)$, which implies that $f(n)=O(\log n)$, 
or $$|T|\leq n^{n(1+\lg\lg n/\lg n + O(1/\lg n))},$$ as desired.
\end{proof}

\section{The Lower Bound}

In this section we provide a probabilistic construction of $n$ log-sparse
sets whose sumset contains an arithmetic progression of 
length $n^{(1-o(1))n}.$

\begin{theorem}
\label{thm:lower}
For any $\eps>0$, there is some positive $n_0=n_0(\eps)$ so that for all
$n \ge n_0(\eps)$, there exists log-sparse $S_i$ for $1 \le i \le n$
so that the sumset $S=S_1+S_2+ \cdots +S_n$ contains 
an arithmetic progression
of length at least $n^{(1-\eps)^2 n}$.
\end{theorem}

\begin{proof}

Begin by splitting the integers from $0$ to $(1-\eps)^2
n \log n-1$ into $(1-\eps) n$ blocks of $(1-\eps) \log n$
consecutive integers.  Denote the blocks as $b_1, \dots, b_m,$
where $m = (1-\eps) n,$ so \[b_i = \{(i-1)(1-\eps)\log n,
(i-1)(1-\eps)\log n + 1, \dots, i(1-\eps) \log n-1\}.\]
For each $i \le m$, let $B_i$ be the set of all positive integers
which are sums of distinct powers of $2$ with exponents in $b_i$.
Then, $|B_i| = 2^{(1-\eps) \log n}-1 = n^{1-\eps}-1.$  Furthermore,
every integer from $0$ to $2^{(1-\eps)^2 n \log n} -1 = n^{(1-\eps)^2
n}-1$ can be uniquely written as the sum of at most one element from
each $B_i,$ by just looking at the integer's binary representation
and splitting it into blocks of size $(1-\eps) \log n.$

We first create sets $S_1, \dots, S_n$, each of size $m+1 = (1-\eps) n
+ 1.$ For each $1 \le i \le m$ and each $1 \le j \le n,$ we uniformly
at random choose one element in $B_i$ to be in $S_j$. Also, allow
each $S_j$ to contain $0$.	This is not important since at the end
we can shift all the elements of each $S_i$ up by $1$, and clearly
there are at most $2$ elements in $[x, n^{1-\eps} x) \supset [x, 2x)$
for each integer $x$, both before and after the shift. Therefore,
we have that each $S_i$ is log-sparse.

We show that with positive probability, $[0, 2^{(1-\eps)^2 n \log n})
\subset S,$ which clearly concludes the proof. For an integer $0
\le a < 2^{(1-\eps)^2 n \log n}$, write $a$ as $x_1+\dots+x_m,$
where $x_i \in B_i \cup \{0\}$. Consider a bipartite 
graph $G$ with nodes $x_1,
\dots, x_m$ and $S_1, \dots, S_n$ such that there is an edge from
$x_i$ to $S_j$ if and only if $x_i \in S_j.$ Then, suppose that
for any $k \le m$ and $1 \le i_1 < \dots < i_k \le m$, there exist
$k$ integers $1 \le j_1 < \dots < j_k \le n$ such that $S_{j_r}$
contains some $x_{i_t}$ for all $r \le k$. This implies that for any
subset $\{x_{i_1}, \dots, x_{i_k}\},$ the total number of $S_j$'s
that some $x_{i_t}$ is connected to in $G$ is at least $k$. Therefore,
by Hall's marriage theorem, there is some matching from $x_1, \dots,
x_m$ to $S_1, \dots, S_n$, i.e., there is a permutation $\sigma:
[n] \to [n]$ such that $x_i \in S_{\sigma(i)}$ for all $i \le m$,
and thus, $a = x_1 + \dots + x_n \in S_1 + \dots + S_n.$

Therefore, it suffices to show that the probability of there existing some
$1 \le k \le m$, some subset $\{B_{i_1}, \dots, B_{i_k}\} \subset
\{B_1, \dots, B_m\}$, some $x_{i_1} \in B_{i_1}, \dots, x_{i_k} \in
B_{i_k}$, and some $\{S_{j_1}, \dots, S_{j_{n-k+1}}\} \subset \{S_1,
\dots, S_n\}$ such that no $x_{i_t}$ is contained in any $S_{j_r}$,
is less than $1$.  This follows from the union bound.  We can
upper bound the probability by at most
\[ \sum\limits_{k = 1}^m {m \choose k} \cdot (n^{1-\eps})^k \cdot {n
\choose n-k+1} \cdot \left(1 - \frac{1}{n^{1-\eps}}\right)^{k(n-k+1)} \]
The ${m \choose k}$ comes from choosing the subset
$\{B_{i_1}, \dots, B_{i_k}\}$, the $(n^{1-\eps})^k$ comes
from choosing each $x_{i_r},$ the ${n \choose n-k+1}$
comes from choosing the $S_{j_t}$'s and the $\left(1 -
\frac{1}{n^{1-\eps}}\right)^{k(n-k+1)}$ is the probability that
every $S_{j_r}$ does not contain any $x_{i_t}.$

Now, using the fact that ${m \choose k} \le m^k \le n^k$ and ${n
\choose n-k+1} \le n^{k-1} \le n^k,$ this sum is at most
\[\sum\limits_{k = 1}^{m} \left(n^2 \cdot n^{1-\eps} \cdot \left(1 -
\frac{1}{n^{1-\eps}}\right)^{n-k+1}\right)^{k}.\]
But since $k \le (1-\eps)n,$ we know that $n-k+1 \ge \eps n,$
so this sum is at most
\begin{align*} \sum\limits_{k = 1}^{m} \left(n^2 \cdot n^{1-\eps}
\cdot \left(1 - \frac{1}{n^{1-\eps}}\right)^{\eps n}\right)^{k} &\le
\sum\limits_{k = 1}^{m} \left(n^2 \cdot n^{1-\eps} \cdot e^{-\eps
n^{\eps}}\right)^{k} \\ &\le \sum\limits_{k = 1}^\infty \left(n^3 \cdot
e^{-\eps n^{\eps}}\right)^k < 1, \end{align*}
assuming $n$ is sufficiently large.  This concludes the proof.
\end{proof}

\section{Explicit construction}

The proof of the lower bound above is probabilistic. It is not
difficult to derandomize this proof and give an explicit 
construction containing a progression of length
$2^{\Omega (n \log n)}$ using quadratic polynomials
over a finite field. The construction is described
in what follows.  It is possible to use other 
known explicit bipartite graphs known as
condensers to get similar constructions,
but the one below is probably the simplest to describe.
See, e.g., \cite{TU} and its references for some more sophisticated
constructions of condensers.

Let $F=F_q$ be the finite field of size $q$. Define a bipartite
graph
$G=G_q$ with classes of vertices $A$ and $B$ as follows.
$A=F \times F$ is simply the cartesian product of $F$ with itself.
$B$ is the disjoint union of $q^2$ sets $B_{a,b}$ with $a,b \in F$.
Each set $B_{a,b}$ consists of the $q$ polynomials
$P_{a,b,c}(x)=ax^2+bx+c$  where $c$ ranges over all elements of
$F$.
Each vertex $P=P_{a,b,c} \in B$ is connected to all vertices
$(x,P(x))\in A$. Therefore, the degree of each vertex in $B$ is
exactly
$q$. Note that for every fixed $a,b$, the sets of neighbors
of the $q$ vertices $P_{a,b,c}$ as $c$ ranges over all elements of
$F$ are pairwise disjoint, and each
vertex of $A$ is connected to exactly one of them.
\begin{proposition}
\label{p12}
Let $G=G_q$, $A$ and $B$ be as above. Then for
every $x \leq q^2/4$, every set of at most $x$ vertices
of $B$ has at least $x$ neighbors in $A$. Therefore, for each such
subset of $x$ vertices in $B$ there is a matching in $G$ saturating it, i.e., each vertex in the subset of $B$ is matched.
\end{proposition}
\begin{proof}
Every two vertices of $B$ have at  most $2$ common neighbors in
$A$, since
any two distinct quadratic polynomials can be equal on at most $2$
points. Therefore, if $x \leq (q+1)/2 $
then for every set $X \subset B$
of size $|X|=x$, the number of its neighbors in $A$ is at least
$$
q+(q-2)+(q-4)+ \ldots +(q-2x+2) =x(q-x+1).
$$
This is (much) larger than $x$ for all $x \leq (q+1)/2$. For
$x=\lfloor (q+1)/2 \rfloor$ this number exceeds $q^2/4$, implying
that every set of at least $\lfloor (q+1)/2 \rfloor$ vertices
of $B$ has more than $q^2/4$ neighbors, completing the proof.
\end{proof}

Returning to our sumset problem, put $n=q^2$.
Split the integers in
$[0,q^2 \log q/4)$ into $q^2/4$ blocks, each of size
$\log_2 q$. Each set $S_i$ contains, as in the probabilistic proof, 
the integer $0$ and one sum of the powers of $2$ corresponding
to each block. The assignment is determined by the 
induced subgraph of the graph
$G_q$ described above on the classes of vertices $A$ and
the union of some $q^2/4$ subsets $B_{a,b}$.
The proposition ensures that
$S_1+\ldots +S_n$ contains all integers from $0$ to
$2^{q^2 \log q/4}=2^{n \log n /8}$.
\vspace{0.2cm}

\noindent
{\bf Acknowledgement:}\, We thank Christian Elsholtz for helpful comments.

\end{document}